\documentclass[12pt]{amsart}
\usepackage{amsmath}
\usepackage{amssymb,amscd}
\usepackage[T1]{fontenc}
\usepackage[latin1]{inputenc}
\usepackage[colorlinks=true, urlcolor=blue,bookmarks=true,bookmarksopen=true, citecolor=blue,hypertex]{hyperref}

\numberwithin{equation}{section}

\newtheorem{defn}{Definition}[section]
\newtheorem{theorem}{Theorem}[section]

\newtheorem{lemma}[theorem]{Lemma}

\newtheorem{remark}[theorem]{Remark}

\def \begineq{\begin{equation}}
\def \endeq{\end{equation}}

\def \bb{\mathbb}

\def \CC{{\bb{C}}}

\def \QQ{{\bb{Q}}}
\def \RR{{\bb{R}}}

\def \ZZ{{\bb{Z}}}

\def \({\left(}
\def \){\right)}
\def \<{\langle}
\def \>{\rangle}
\def \bar{\overline}

\begin{document}
\title[Mckay corespondence in Quasitoric orbifolds]{ Mckay corespondence in  Quasitoric orbifolds}
\author[Saibal Ganguli]{Saibal Ganguli}
\address{ Departamento de
Matem\'aticas, Universidad de los Andes, Bogota, Colombia}
\email{saibalgan@gmail.com}


\abstract
{ We show Mckay correspondence of Betti numbers of Chen-Ruan cohomology for omnioriented quasitoric orbifolds. In previous articles with M. Poddar \cite{[GP]}, \cite{[GP2]}, we proved the correspondence for four dimension and six dimensions. Here we deal with the general case.}

\endabstract
\maketitle
\section{\bf{Quasitoric orbifolds } }\label{smooth}

 In this section we
review the combinatorial construction  of quasitoric
orbifolds. We also  construct an explicit
orbifold atlas for them and list a few important properties.
 The notations established here will be important for
the rest of the article.This material has been taken from \cite{[GP2]}

\subsection{\bf {Construction} }   Fix a copy $N$ of $\ZZ^n$ and let $ T_N := (N \otimes_{\ZZ} \RR) / N \cong \RR^n/ N $
 be the corresponding $n$-dimensional torus. A primitive vector in $N$, modulo sign,
  corresponds to a circle subgroup
 of $T_N$. More generally, suppose $M$ is a submodule of $N$ of rank $m$. Then
 \begin{equation}\label{TM} T_M := (M \otimes_{\ZZ} \RR) /M \end{equation}
  is a torus of dimension $m$.
  Moreover there is a natural homomorphism of Lie
 groups $\xi_M: T_M \to T_N$ induced by the inclusion $M \hookrightarrow N$.
\begin{defn}\label{tlambda}  Define T(M) to be the image
  of $T_M$ under $\xi_M$.
 If $M$ is generated  by a vector $\lambda \in N$, denote $T_M$
 and $T(M)$ by $T_{\lambda}$ and $T(\lambda)$ respectively.
\end{defn}

Usually a polytope is defined to be the convex hull of a finite
set of points in $\RR^n$. To keep our notation manageable, we will
take a more liberal interpretation of the term polytope.
\begin{defn}\label{polytope} A polytope $P$ will denote a subset of $\RR^n$ which
is diffeomorphic, as manifold with corners, to the convex hull $Q$
of a finite number of points in $\RR^n$. Faces of $P$ are the
images of the faces of $Q$ under the diffeomorphism.
\end{defn}

Let $P$ be a simple polytope in $\RR^n$, i.e. every vertex of $P$
is the intersection of exactly $n$ codimension one faces (facets).
 Consequently every $k$-dimensional face $F$ of $P$ is the intersection of
 a unique collection of $n-k$ facets.
 Let $ \mathcal{F}:= \{F_1,\ldots, F_m\}$ be the
set of facets of $P$.

\begin{defn}\label{charfn}
 A function $\Lambda: \mathcal{F} \to N $
  is called a characteristic function for $P$ if
 $ \Lambda(F_{i_1}), \ldots, \Lambda(F_{i_k})$ are linearly independent whenever
 $F_{i_1}, \ldots, F_{i_k}$intersect at a face in $P$. We write $\lambda_i$
 for $\Lambda(F_i)$ and call it a characteristic vector.
\end{defn}

\begin{remark}\label{primi} In this article we  assume that all characteristic
 vectors are primitive.
Corresponding quasitoric orbifolds have been termed primitive
quasitoric orbifold in \cite{[PS]}. They are characterized by the
codimension of singular locus being greater than or equal to four.
\end{remark}

 \begin{defn}\label{NF}
 For any face $F$ of $P$, let $\mathcal{I}(F) = \{ i  | F \subset F_i  \}$.
 Let $\Lambda$ be a characteristic function for P.
  The set $ \lambda_F :=\{ \lambda_i:  i \in \mathcal{I}(F) \} $ is called the characteristic set of $F$.
 Let $N(F)$ be the submodule of $N$ generated
by $\lambda_F$. Note that $\mathcal{I}(P) $ is empty  and
$N(P) = \{0\}$. \end{defn}

  For any point $p \in P$, denote by $F(p)$
 the face of  $P$ whose relative interior contains $p$. Define an equivalence relation $\sim$ on the space
 $P \times T_N$ by
 \begin{equation} \label{defequi}
 (p,t) \sim (q,s) \; {\rm if \; and\; only\; if\;} p=q \; {\rm and} \; s^{-1}t \in T(N({F(p)}))
 \end{equation}
 Then the quotient space $X :=P \times T_N/ \sim$ can be given the structure of a $2n$-dimensional
   orbifold. We refer to the pair $(P,\Lambda)$ as a model for the
  quasitoric orbifold. The space $X$ inherits an action of $T_N$ with orbit space
   $P$ from the natural action
  on $P \times T_N$. Let $\pi: X \to P$ be the associated quotient map.

   The space $X$ is a manifold if the characteristic vectors $ \lambda_{i_1}, \ldots, \lambda_{i_k}$
   generate a unimodular subspace of $N$ whenever the facets $F_{i_1}, \ldots, F_{i_k}$ intersect.
    The points $\pi^{-1}(v) \in X$, where $v$ is any vertex of $P$,
are fixed by the action of $T_N$. For simplicity we will denote
the point $\pi^{-1}(v)$  by $v$ when there is no confusion.

\subsection{\bf{ Orbifold charts}}\label{diffs}
Consider open
 neighborhoods $U_v \subset P$ of the vertices $v$  such that
  $U_v$ is the complement in $P$ of all edges
 that do not contain $v$.
 Let
 \begin{equation}
X_v := \pi^{-1}(U_v) =  U_v \times T_N / \sim
\end{equation}
For a face $F$ of $P$ containing $v$ there is a natural inclusion of
$N(F)$ in $N(v)$.
  It induces an injective homomorphism $T_{N(F)} \to T_{N(v)}$ since a basis of $N(F)$ extends
  to a basis of $N(v)$. We will regard $T_{N(F)}$ as a
  subgroup of $T_{N(v)}$ without confusion.
 Define an equivalence relation $\sim_v$ on $U_v \times T_{N(v)}$ by
$(p,t)\sim_v (q,s)$ if
 $p=q$ and $s^{-1}t \in T_{N(F)}$ where $F$ is the face whose relative interior contains $p$.
Then the space
\begin{equation}
 \widetilde{X}_v:= U_v \times T_{N(v)}/ \sim_v
\end{equation}
is $\theta$-equivariantly diffeomorphic to an open set in $\CC^n$,
where  $\theta: T_{N(v)} \to U(1)^n$ is an isomorphism, see \cite{[DJ]}. There
     exists a diffeomorphism $f: \widetilde{X}_v \to B \subset \CC^n  $ such that
   $f(t\cdot x) = \theta(t) \cdot f(x)$ for all $x \in \widetilde{X}_v $. This will be
   evident from the subsequent discussion.

The map $\xi_{N(v)} : T_{N(v)} \to T_N$ induces a map $\xi_v:
\widetilde{X}_v \to X_v$ defined by
   $\xi_v([(p,t)]^{\sim_v}) = [(p,\xi_{N(v)}(t)) ]^{\sim}$ on equivalence classes.
    The kernel of $\xi_{N(v)}$,  $G_v = N/N(v)$, is a finite
  subgroup of $T_{N(v)}$ and therefore has a natural smooth, free action on $T_{N(v)}$
  induced by the group operation.  This induces smooth action of $G_v$ on
  $\widetilde{X}_v$. This action is not free in general. Since $T_N \cong T_{N(v)}/G_v $,  $X_v$
 is homeomorphic to the quotient space $\widetilde{X}_v/G_v$.  An orbifold chart
 (or uniformizing system) on $X_v$ is
  given by $(\widetilde{X}_v, G_v, \xi_v)$.

% Let $(p_1, \ldots, p_n)$ denote the standard
 %   coordinates on  $\RR^n \supset P$.
 % Let $q_1, \ldots, q_n$ be the coordinates on $N \otimes \RR$ with respect to the standard
 % basis of $N$. They correspond to standard angular coordinates on $T_N$.
 %Let $\{u_1, \ldots, u_n\}$ be the standard basis of $N$.
 %Suppose the characteristic vectors $u_i$ are assigned to the facets $p_i=0$ of the
 %cone $ \RR^n_{\ge} $.
  %   In this case  there is a homeomorphism $\phi: (\RR^n_{\ge}
   %  \times T_N/\sim ) \to  \RR^{2n}$
   % given by
    % \begin{equation}
  % x_i= \sqrt{p_i} \cos (2 \pi q_i), \;\;
  %  y_i = \sqrt{p_i} \sin (2 \pi q_i) \;\; {\rm where} \; i=1 \ldots n.
   % \end{equation}

      We  define a homeomorphism
      $\phi(v): \widetilde{X}_v  \to \RR^{2n}$ as follows.
      Assume without loss of generality that $F_1, \ldots, F_n$ are the facets of
      $U_v$. Let the equation of $F_i$ be $p_i(v)= 0$.
     Assume that $p_i(v) > 0$ in the interior of $U_v$ for every $i$.
 Let $\Lambda_{(v)}$ be the corresponding matrix of characteristic
 vectors \begin{equation}\label{lambdav}
 \Lambda_{(v)}= [\lambda_{1} \ldots \lambda_{n} ]. \end{equation}

     If ${\bf q}(v)= (q_1(v),\ldots, q_n(v))^t$ are angular coordinates of an element of $T_N$ with respect
    to the basis $\{ \lambda_1, \ldots, \lambda_n \}$ of $N \otimes \RR$, then  the
     standard coordinates ${\bf q} =(q_1, \ldots, q_n)^t$ may be expressed as
     \begin{equation}\label{chcoord1}
     {\bf q} = \Lambda_{(v)} {\bf q}(v).
     \end{equation}
     Then define the homeomorphism $\phi(v): \widetilde{X}_v \to \RR^{2n}$ by
     \begin{equation}\label{homeo}
      x_i = x_i(v):= \sqrt{p_i (v)} \cos(2 \pi q_i(v) ), \quad
      y_i = y_i(v):= \sqrt{p_i(v) } \sin( 2 \pi q_i(v) ) \quad {\rm for}\;
      i=1,\ldots,n
     \end{equation}
		
		\begin{remark} The square root over $p_i$ is necessary to ensure that the
orbit map  is smooth.
\end{remark}
		
We write
\begin{equation}\label{cplxcoor}
 z_i = x_i + \sqrt{-1} y_i, \quad {\rm and} \quad
z_i(v) = x_i(v) + \sqrt{-1}y_i(v)
\end{equation}

 Now consider the action of $G_v = N/N(v)$ on $\widetilde{X}_v$. An element
     $g$ of $G_v$ is represented by a vector $\sum_{i=1}^n a_i \lambda_i $ in
   $N$ where each $a_i \in  \QQ$.  The action of $g$ transforms the coordinates
   $q_i(v)$ to $q_i(v) + a_i$. Therefore
    \begin{equation}\label{action}
    g\cdot (z_1,\ldots, z_n) = (e^{2\pi \sqrt{-1} a_1} z_1,\ldots, e^{2\pi \sqrt{-1}a_n} z_n).
     \end{equation}

We define
\begin{equation}\label{gx3}
  G_F := ((N(F)\otimes_{\ZZ} \QQ) \cap N) / N(F).
\end{equation}
We denote the space X with the above orbifold structure by
$\bf{X}$.

\subsection{\bf{Invariant suborbifolds}} The $T_N$-invariant subset $X(F) = \pi^{-1}(F)$, where
    $F$ is a face  of $P$, has a natural structure of a quasitoric orbifold \cite{[PS]}.
     This structure is
    obtained by taking $F$ as the polytope for ${\bf X}(F)$ and projecting the characteristic vectors
     to $N/N^{\ast}(F)$ where $N^{\ast}(F)= (N(F) \otimes_{\ZZ} \QQ) \cap N$. With this structure
    ${\bf X}(F)$ is a suborbifold of ${\bf X}$.
    It is called a characteristic suborbifold if $F$ is a facet.
    Suppose $\lambda$ is the characteristic vector attached to the facet $F$. Then
     $\pi^{-1}(F)$ is fixed by the circle subgroup $T(\lambda)$ of $T_N$.
      We denote the relative interior
      of a face $F$ by $F^{\circ}$ and the corresponding invariant space $\pi^{-1}(F^{\circ})$ by $X(F^{\circ})$.
      Note that $v^{\circ}= v$ if $v$ is a vertex.

\subsection{\bf{Orientation}}
       Quasitoric orbifolds are oriented. For  more detailed dicussion see 2.8
			\cite{[GP2]}.
       A choice of orientations for $P \subset \RR^n$
      and $T_N$ induces an orientation for ${\bf X}$.

\subsection{\bf{Omniorientation}}\label{omnio} An omniorientation is a choice of orientation for the
  orbifold as well as an orientation for each characteristic suborbifold.
  At any vertex $v$, the $G_v$-representation ${\mathcal T}_0 \widetilde{X}_v$ splits into the direct
  sum of $ n$  $G_v$-representations corresponding to the normal spaces of $z_i(v)=0$.
  Thus we have a decomposition of the orbifold tangent space
  ${\mathcal T}_v{\bf X}$ as a direct sum of the normal spaces of the
  characteristic suborbifolds that meet at $v$.
   Given an omniorientation, we say that the
  sign of a vertex $v$ is positive if the orientations of
   ${\mathcal T}_v({\bf X})$ determined by the
  orientation of ${\bf X}$ and orientations of characteristic suborbifolds coincide. Otherwise
  we say that sign of $v$ is negative. An omniorientation is then said to be positive
   if each vertex has positive sign.

 It is easy to verify that reversing the sign of any number of characteristic vectors does not
 affect the topology or differentiable structure of the quasitoric orbifold. There is a circle
 action of $T_{\lambda_i}$ on the normal bundle of ${\bf X}(F_i)$ producing a complex structure and
 orientation on it. This action and orientation varies with the sign of $\lambda_i$. Therefore,
 given an orientation on ${\bf X}$, omniorientations correspond bijectively to choices of signs
 for the characteristic vectors. We will assume the standard orientations on $P$ and $T^n$ so that
 omniorientations will be solely determined by signs of characteristic vectors. Also under this
 choice the sign of $v$ equals the sign of $\det(\Lambda_{(v)})$.

\section{\bf{ Blowdowns}}\label{blowdown}
This material has been taken from \cite{[GP2]}.
 Topologically the blowup will correspond to replacing an
 invariant suborbifold by the projectivization of its normal bundle. Combinatorially
 we replace a face by a facet with a new characteristic vector. Suppose $F$ is a face of $P$. We
 choose  a hyperplane $H = \{ \widehat{p}_0 = 0 \}$ such that $\widehat{p}_0$
 is negative on $F$ and $\widehat{P}:=\{\widehat{p}_0 > 0\} \cap P$ is a simple
 polytope having one more facet than $P$. Suppose $F_1, \ldots, F_m$ are the
 facets of $P$. Denote the facets $F_i \cap \widehat{P} $ by $F_i$
 without confusion. Denote the extra facet $H \cap P$ by
 $F_{0}$.

Without loss of generality let $F = \bigcap_{j= 1}^k F_j$.
  Suppose there exists a primitive vector
  $\lambda_{0} \in N$ such that
  \begin{equation}
 \lambda_{0} = \sum_{j= 1}^k b_j \lambda_j, \; b_j > 0 \, \forall
 \, j.
  \end{equation}
 Then the assignment $F_{0} \mapsto \lambda_{0}$
 extends the characteristic function of $P$ to a characteristic function
 $\widehat{\Lambda}$ on $\widehat{P}$. Denote the omnioriented quasitoric
 orbifold derived from the model $(\widehat{P}, \widehat{\Lambda}  )$ by
 ${\bf Y}$.
\begin{defn}
We define  blowdown a map $\rho$ from ${\bf Y} \mapsto {\bf X}$  which is inverse of a blow-up. Such maps have beeen constructed in   \cite{[GP2]}.
\end{defn}

\begin{lemma} (Lemma 4.2  \cite{[GP2]}) If ${\bf X}$ is
positively omnioriented, then so is a blowup ${\bf Y}$.
\end{lemma}

\begin{defn}\label{crepant}
A blowdown is called crepant if $\sum b_j = 1 $.
\end{defn}

\section{\bf{ Chen-Ruan Cohomology}}\label{crc}
This material has been taken from \cite{[GP2]}.
The Chen-Ruan cohomology group is built out of the ordinary
cohomology of certain copies of singular strata of an orbifold
called twisted sectors. The twisted sectors of orbifold toric
varieties was computed in \cite{[Po]}. The determination of such
sectors for quasitoric orbifolds is similar in essence. Another
important feature of Chen-Ruan cohomology is the grading which is
 rational in general. In our case the grading will  depend on
 the omniorientation.

 Let ${\bf X}$ be an omnioriented
quasitoric orbifold. Consider any element $g$ of the group $G_F $
\eqref{gx3}. Then $g$ may be represented by a vector $\sum_{j \in
\mathcal{I}(F)} a_j \lambda_j $. We may restrict $a_j$ to $[0,1)\cap \QQ$.
Then the above representation is unique. Then define the degree
shifting number or age of $g$ to be
\begin{equation}\label{age}
\iota(g)= \sum a_j.
\end{equation}

For faces $F$ and $H$ of $P$ we write $F \le H$ if $F$ is a
sub-face of $H$, and $F < H$ if it is a proper sub-face. If $F \le H
$ we have a natural inclusion of $G_H$ into $G_F$ induced by the
inclusion of $N(H)$ into $N(F)$. Therefore we may regard $G_H$ as
a subgroup of $G_F$. Define the set
\begin{equation}
G_F^{\circ} = G_F - \bigcup_{F < H} G_H
\end{equation}
Note that $G_F^{\circ} = \{ \sum_{j \in \mathcal{I}(F)} a_j \lambda_j |
0 < a_j < 1 \} \cap N  $, and $G_P^{\circ}= G_P =\{0\}$.

\begin{defn}\label{cr}
 We define the Chen-Ruan orbifold cohomology
 of an omnioriented quasitoric orbifold ${\bf X}$ to be
$$ H^{\ast}_{CR}({\bf X}, \RR ) =
\bigoplus_{F \le P} \bigoplus_{ g\in G_F^{\circ}} H^{\ast - 2
\iota(g)} (X(F), \RR).$$
 Here $H^{\ast}$ refers to singular cohomology or equivalently to
 de Rham cohomology of invariant forms when $X(F)$ is considered
  as the orbifold ${\bf X}(F)$.
 The pairs $(X(F), g)$ where $F<P$  and $ g\in G_F^{\circ}$ are
called twisted sectors of ${\bf X}$. The pair $(X(P),1)$, i.e. the
underlying space $X$, is called the untwisted sector.
\end{defn}
First we introduce some notation. Consider a codimension $k$ face
$F= F_1 \cap \ldots \cap F_k$ of $P$ where $k \ge 1$.
  Define a $k$-dimensional cone $C_F$ in $N\otimes \RR$ as follows,
\begin{equation}\label{cf}
C_F = \{ \sum_{j=1}^k a_j \lambda_j: a_j \ge 0 \}
\end{equation}
 The group $G_F$ can be identified with the subset $Box_F $ of
 $C_F$, where
\begin{equation}\label{boxf}
 Box_F := \{
\sum_{j=1}^k a_j \lambda_j: 0\le a_j <1 \} \cap N.
\end{equation}
Consequently the set $ G_F^{\circ}$ is identified with the subset
\begin{equation}\label{boxfo}
 Box_F^{\circ} :=  \{ \sum_{j=1}^k a_j
\lambda_j: 0 < a_j < 1 \} \cap N \end{equation} of the interior of
$C_F$. We define $Box_P = Box_P^{\circ}= \{0\} $.

 Suppose $v=F_1\cap \ldots \cap F_n$ is a vertex of $P$. Then
$Box_v =  \bigsqcup_{v\le F} Box_{F}^{\circ} $. This implies
\begin{equation}\label{Gdecom}
G_v = \bigsqcup_{v \le F} G_F^{\circ}
\end{equation}
An almost complex orbifold is $SL$ if the linearization of each
$g$ is in $SL(n,\CC)$. This is equivalent to $\iota(g)$ being
integral for every twisted sector.
 Therefore, to suit our purposes, we make the following definition.

\begin{defn}\label{quasisl}
 A quasitoric
orbifold is said to be quasi-$SL$ if the age of every twisted
sector is an integer.
\end{defn}

\begin{lemma} (Lemma 8.2  \cite{[GP2]}) The crepant blowup of a quasi-$SL$ quasitoric orbifold is
quasi-$SL$.
\end{lemma}

%\begin{proof} Suppose the blowup is along a face $F= F_1 \cap \ldots \cap F_k$.
% The new sectors that appear correspond to $G_H^{\circ}$
%where $H < F_0$. Take any vertex $v$  in $H$. Suppose $v$ projects
%to the vertex $w$ of $F$ under the blowdown.  Without loss of
%generality assume $w = \bigcap_{j=1}^n F_j $. Then $v =
%\bigcap_{0\le j\neq i \le n} F_j $ for some
 %$1 \le i \le k$. Without loss of generality assume $i=1$.
 %Since $v \le H $, $\mathcal{I}(H) \subset \{0, 2, \ldots, n \}$.
%Therefore any $g\in G_H^{\circ}$ may be represented by an element
%$\eta= c_0 \lambda_0 + \sum_{j=2}^n c_j \lambda_j $ of $N$ where
%each $ c_j \in [0,1)\cap \QQ $. We need to show that the age of
%$g$, namely $c_0 + \sum_{j=2}^n c_j $, is an integer.

 %But using $\lambda_0 = \sum_{j=1}^k b_j
%\lambda_j $ we get that $\eta \in C_w$. In fact
%\begin{equation}
%\eta = c_0 b_1 \lambda_1 + \sum_{j=2}^k (c_0 b_j + c_j)\lambda_j +
%\sum_{j=k+1}^n c_j \lambda_j
%\end{equation}
%We may write $\eta= \sum_{j=1}^n (m_j + a_j)\lambda_j$ where each
%$m_j$ is an integer and each $a_j \in [0,1)\cap \QQ$. Then
%$\sum_{j=1}^n a_j \lambda_j $ corresponds to an element of $G_w$.
%Since ${\bf X}$ is quasi-$SL$, $\sum_{j=1}^n a_j$ must be an
%integer. Therefore $\sum_{j=1}^n (m_j + a_j) $ is an integer.
%Hence $c_0 b_1 + \sum_{j=2}^k (c_0 b_j + c_j) + \sum_{j=k+1}^n c_j
%$ is an integer. Using $\sum_{j=1}^k b_j = 1$, this yields that
%$c_0 + \sum_{j=2}^n c_j$ is an integer.
%\end{proof}
%\newpage
\section{\bf{Correspondence of Betti numbers}}\label{Mckay}

\subsection{Singularity and lattice polyhedron}
Following the discussions in sections $\ref{crc}$ , a singularity of a
face F is defined by a cone $ C_{F}$ formed by positive linear
combinations of vectors in its characteristic  set ${\lambda_{1},\ldots
,\lambda_{d}}$ where  d is the codimension of the face in the
polytope. The elements of the  local  group $ G_F$ are of the form
$g =diag(e^{2\pi \sqrt{-1} \alpha_{1}},\ldots,e^{2\pi \sqrt{-1}
\alpha_{d}}),$  where $ \sum _{i=1}^{d}\alpha_{i}\lambda_{i} \in
N $, and $ 0 \leq \alpha_{i} <1 $. Recall that the age
\begin{equation}
\iota(g) = \alpha_{1} + \ldots + \alpha_{d}
\end{equation}
is integral in quasi-$SL$ case  by definition $\ref{quasisl}$.

The singularity along the interior of $F$ is of the form $\CC^d/G_F$.
These singularities  are same as Gorenstein toric quotient singularities in complex  algebraic  geometry. Now let $N_{v}$ be the  lattice formed by
 $\{\lambda_{1},\ldots ,\lambda_{n}\}$, the characteristic vectors at a vertex $v$  contained in the face $F$. Let $m_{v}$  be the element in the dual lattice of
$N_v$ such that its  evaluation on each $\lambda_{i}$ is one.
 Now from  Lemma 9.2  of \cite{[CP]} we know that the cone $ C_{v}$ contains an integral basis, say $e_1,\ldots,e_n$. Suppose $e_i= \sum a_{ij} \lambda_j$. By \eqref{boxf}
$e_i$ corresponds to an element of $G_v$,
and since the singularity is qausi-$SL$, $\sum a_{ij}$ is integral.
Hence $m_{v}$ evaluated on each $e_j$ is integral. So $m_{v}$ an element of the dual
of the  integral lattice $ N$.

 Consider the $(n-1)$-dimensional lattice polyhedron  $\Delta_v$ defined as  $\{x \in C_v\mid \<$ x $,m_v \>= 1\}$. Note that $\Delta_v =\{ \sum_{i=1}^n a_i \lambda_i \mid a_i \ge 0, \; \sum a_i = 1 \}$.
  For any face $F$  containing $v$  we define $\Delta_F = \Delta_v \cap C_F$.
	If $\{\lambda_i, \ldots, \lambda_d \} $ denote the chracteristic set of $F$, then
	$\Delta_F =\{ \sum_{i=1}^d a_i \lambda_i \mid a_i \ge 0, \; \sum a_i = 1 \} $.
	Hence $\Delta_F$ is independent of the choice of $v$.

%Then L is a full sublattice of $R^{d}= L_{R}$, $ Z_{d}$ of finite index in L and $\frac{L}{Z_{d}}$ is canonically isomorphic to G.Therefore, the cone $\sigma$ defining the affine toric variety $X = A_{\sigma}$ is the
% positive d-dimensional octant$R^{d}_{d \geq 0 }\subset R^{d}=L_{R} $$(L\otimes R)$. We call such singularities $ \bf{toric} $ $\bf{singularity}$.Define  $ m_{\sigma} \in  M $  to be$ ( 1,1, \ldots 1) \subset Z^{d} $. For any $x \in  \sigma \cap N $,we define a degree $ deg x := <x,m_{\sigma}>$,  .Such that
% \begin{equation}
% n_ {1}=(1,0 \ldots 0) ,n_ {2}=(0,1 \ldots 0)\ldots n_ {d}=(0,0 \ldots ,0,1 )
% \end{equation}
  %form a  sequence of degree 1 .
 %For this reason
  %The element $ (\alpha_{1}, \alpha_{2} \ldots \alpha_{d}) \in N $ corresponds precisely to the element $ g=diag(e^{2\pi \sqrt{-1} \alpha_{1}},e^{2\pi \sqrt{-1} \alpha_{2}},\ldots e^{2\pi \sqrt{-1} \alpha_{n}})$.
  %A lattice polyhedron  $\Delta$  is defined as  $\{x \mid \<$ x $,m_\sigma \>= 1\}$ .

%It is clearly seen that
 %\begin{equation}
 %\iota(g)= <( \alpha_{1}, \ldots,\alpha_{d}),m_{F}>
 %\end{equation}
 %where $m_F$ is the restriction of $m_v$ to integral points in $C_F$.

 %Thus  the Poincare series of the quotient ring $\frac{S}{n_{1},n_{2} \ldots ,n_{d}}$ equals
 %\begin{equation}
 %\psi_ {0}(\Delta)  + \psi_ {1}(\Delta)t \ldots \psi_ {d-1}(\Delta) t ^{d-1}
 %\end{equation}
 %We define
  %\begin{equation}
  %\psi_{i}(G_F)=\#\{ g \in G_F \mid \iota(g)= i \}
  %\end{equation}
 \begin{remark}
 An element  $g \in G$ of an $SL$ orbifold  singularity can be diagonalized to the form g$=diag(e^{2\pi \sqrt{-1} \alpha_{1}},\ldots,e^{2\pi \sqrt{-1} \alpha_{d}}),$ where $0 \leq  \alpha_i < 1$ and
$\iota(g)= \alpha_{1} + \ldots +\alpha_{d}$ is integral.
\end{remark}

We make some definitions following \cite{[BT]}.
\begin{defn}
Let $G$ be a finite subgroup of $SL(d,\mathbb{C})$.  Denote  by $ \psi_ {i}(G)$ the
number of the conjugacy classes of $G$ having  $\iota(g)=i$. Define
\begin{equation}\label{lat1}
 W(G;uv)=\psi_ {0}(G)  + \psi_ {1}(G)uv + \ldots +\psi_ {d-1}(G) {(uv)} ^{d-1}
 \end{equation}
 \end{defn}

\begin{defn}\label{ht}
 We   define height(g) = rank(g-I)
\end{defn}

\begin{defn}
  Let $G$ be a finite subgroup of $SL(d,\mathbb{C})$. Denote by $\tilde{\psi_ {i}}(G)$
   the number of the conjugacy classes of $G$ having the $height =d$  and $\iota(g)=i$.
 \begin{equation}\label{lat2}
 \widetilde{W}(G;uv)=\tilde{\psi}_{0}(G)  + \tilde{\psi}_{1}(G)uv + \ldots  + \tilde{\psi}_{d-1}(G) {(uv)} ^{d-1}
 \end{equation}
 \end{defn}

\begin{defn}
For a lattice polyhedron  $\Delta_F$ defining a $SL$ singularity $\CC^d/ G_F$, we define the following:
\begin{equation}
W(\Delta_F; uv) =W(G_F; uv)
\end{equation}
\begin{equation}\label{psi}
\psi_ {i}(\Delta_F)=\psi_ {i}(G_F)
\end{equation}
\begin{equation}\label{psi1}
\widetilde{W}(\Delta_F; uv) =\widetilde{W}(G_F; uv)
\end{equation}
%\begin{equation}
%\psi_ {i}(\Delta_F)=\psi_ {i}(\G_F)
%\end{equation
\begin{equation}\label{tpsi}
\tilde{\psi}_{i}(\Delta_F)=\tilde{\psi}_{i}(G_F)
\end{equation}
\end{defn}
\begin{defn} A finite collection $\tau=\{ \theta \}$ of simplices with vertices in $ \Delta_F \cap N$  is called a triangulation of $ \Delta_F$ if the following properties
are satisfied.
\begin{enumerate}
\item If $\theta^{\prime}$ is a face of $\theta \in \tau$ then $\theta ^{ \prime} \in \tau$
\item The intersection of any two simplices $ \theta ^{\prime}$, $ \theta ^{\prime \prime} \in \tau$ is either empty, or a common
face of both of them;
\item $\Delta_F = \cup_{\theta \in \tau } \theta $
\end{enumerate}
\end{defn}
%Every triangulation $\tau$ of gives a partial crepent desingularization of $X_{\tau } \rightarrow A_{\sigma} $(in our languange crepent blowdown), of  $ A_{\sigma}$ so %
%that $ X_{\tau }$ has an quasi-SL
%orbifold  singularity.

%\begin{remark}\label{crepant}

\subsection{Blowdown and triangulation of polyhedron}
A crepant blowup gives rise to triangulation of the polyhedrons
corresponding to some of the faces. Suppose we blow up about a
face $F$. Then it is clear that new characteristic vector is an
integral vector lying in the interior  of the polyhedron $
\Delta_F$. Note that $ \Delta_F$ is a simplex. The
crepant blow up induces a barycentric subdivision of  $ \Delta_F$
with the new characteristic vector as barycenter. We denote this
triangulation of $ \Delta_F$  by $\tau_F$. For the  faces
$F^{\prime}$ contained in F,  $ \Delta_{F^{\prime}}$  is
triangulated  as follows. Let
  $K_{F^{\prime}}=\lambda_{F^{\prime}}-\lambda_{F}$ be  difference of two characteristic sets. The triangulation $\tau_{F^{\prime}}$ consists of  simplices with vertex set of the form
    $\theta \cup \beta$ where $ \theta$ are the vertices of a  simplex of $\tau_F $ and $\beta \subset K_{F^{\prime}}$. To see that this process takes  care of all the faces lost and created we make the following comments. First of all the faces lost are $F$ and its subfaces. This means there will be no simplex with vertex set having $\lambda_{F}$ as a subset. This is exactly what happens here. The new faces created are subfaces of the intersection of new facet (created by the blowup)  with faces having as vertex one of the vertices  of  $F$. These  faces intersected F prior to the blow up in  some $F^{\prime}$ and so the new faces  formed correspond to the simplices with vertex set that are subset of the union $\theta \cup \beta$ discussed above.  %Now consider a chain of $ F \supset F_1 \ldots \supset F_{n-d} $ ( where d is codimension of the face F  and $F_i$  are faces and the difference of  characteristic ssets of two consecutive elements in the chain is one characteristic vector). The blow up creates a triangulation
%in  $ \Delta_ {F_i}$ in the following manner.$ \Delta_F$  is triangulated by above process so a simplex $\theta$  a $\bf{highest}$ $\bf{dimensional}$ $\bf{simplex}$ is in $\tau_0 $  the triangulation of $ \Delta_F$ (discussed above)  if and only if
%$ \{\lambda_1\} \cup   \theta $ is in $\tau_1 $( where $\tau_1 $ is the the triangulation induced in $ \Delta_{F_1}$  by this algorithm and $ {\lambda_1}$ is the only characteristic vector lying in the difference of the characteristic sets.) Now once we have a triangulation of $ \Delta_{F_1}$  and we have  a triangulation of $ \Delta_{F_2}$ following the same algorithim and we can go on till $ \Delta_{F_{n-d}}$ which is a polyhedron corresponding for a singularity at a vertex . Similarly we can do this for other chains starting a F. It can be easily be seen that this take cares of all  new faces created by the blow up with the new faces corresponding to the new simplices of the triangulation.
%\end{remark}
 %$\bf{We}$  $\bf{call}$ $\bf{these}$ $\bf{singularities}$ $\bf{toric}$ $\bf{singularities}$
\subsection{\bf{ E-polynomial }}
 The following has been taken from the paper of Batyrev and Dais \cite{[BT]}.
 Let $X$ be an algebraic variety  over $ \mathbb{C}$ which is not necessarily compact
or smooth. Denote by $ h^{p,q}(H_c^{k}(X))$ the dimension of the $(p,q)$ Hodge component
of the $k$-th cohomology with compact supports. We define
\begin{equation}
  e^{p,q}(X)=\Sigma _{k \ge 0 } (-1)^{k} h^{p,q}(H_c^{k}(X)).
  \end{equation}
 The polynomial
 \begin{equation}
 E(X; u, v) := \Sigma _{p,q}  e^{p,q}(X) u^{p}v^{q}
\end{equation}
is called $E$-polynomial of X.

 \begin{remark}\label{epq}
 If the Hodge structure is pure, for example  in the case of   smooth projective toric varieties, then the coefficients $ e^{p,q}(X)$
of the $E$-polynomial of X are related to the usual Hodge numbers by $e^{p,q}(X) =
(-1)^{p+q}h^{p,q}(X)$
\end{remark}
 We state the following theorem  without proof.
 \begin{theorem}\label{strata}(Proposition 3.4 in \cite{[BT]}) Let X be a disjoint
union of locally closed subvarieties $X_j$, $j \in J$, where $J \subset \mathbb{N}$.
Then
\begin{equation}\label{strat1}
E(X;u,v)= \Sigma _{j \in J } E(X_ j;u,v)
\end{equation}
\end{theorem}

\subsection{\bf{ Ehrhart power series}} Let $\Delta$ be a lattice polyhedron and $k\Delta:= \{ kx \mid x\in \Delta \}$.
Let $l( k\Delta)$ be the number of  lattice points of $k\Delta$. Then the Ehrhart power series
\begin{equation}\label{lati}
P_{\Delta}(t)= \Sigma_{ (k \geq 0)} l(k\Delta) t^{k}
\end{equation}
%It is well-known( see for instance\cite{[B]},Thm 2.11)  that $P_{\Delta}(t)$ can always be written in the form

\begin{defn}
Let $\Delta_F$ be a (d-1) dimensional  lattice polyhedron defining a d-dimensional  toric  singularity . It is well-known (see, for instance, \cite{[BT]}, Theorem 5.4)  that $P_{\Delta_F}(t)$ can  be written in the form,
\begin{equation}\label{lati1}
P_{\Delta_F}(t) =\frac{\psi_ {0}(\Delta_F)  + \psi_ {1}(\Delta_F)t + \ldots + \psi_ {d-1}(\Delta_F) t ^{d-1}} {(1-t)^{d}}
\end{equation}
where $ \psi_ {0}(\Delta_F), \ldots, \psi_ {d-1}(\Delta_F)$ are  non-negative integers defined in  $\eqref{psi}$.
%\begin{equation}\label{lati2}
%W(\Delta_F; uv) = \psi_ {0}(\Delta_F) + \psi_ {1}(\Delta_F)uv  \ldots + \psi_ {d-1}(\Delta_F){uv}^{d-1}
%\end{equation}
%Define.
%\begin{equation}\label{lati3}
%\widetilde{W}(\Delta; uv) = \tilde{\psi}_ {0}(\Delta)  + \tilde{\psi}_ {1}(\Delta)t + \ldots  + \tilde{\psi}_ {d-1}(\Delta) t ^{d-1}
%\end{equation}
\end{defn}
%where $\tilde{\psi}_ {i}(\Delta)$ are  certain numbers which will be clear from the following remark.
 %\begin{remark}
 %For polyhedrons formed from  toric singularities discussed above.(See(\ref{lat1})
% \begin{equation}\label{lati4}
 %W(\Delta_F; uv) =W(G_F; uv)
 %\end{equation}
% and (See{\ref{lat2}}
% \begin{equation}\label{lati4}
 %\widetilde{W}(\Delta_F; uv) =\widetilde{W}(G_F; uv)
 %\end{equation}
%where  $G_F$ the corresponding group formed by the formed by the singularity a the face F.
%\end{remark}
\subsection{\bf{ More on quasi-SL orbifolds } }
 Let ${\bf X}$ be a compact quasi-SL  $2n$-dimensional quasitoric orbifold. Let Sing($X$) be the set of singular points of $X$. Consider the set $ I=\{i \in  \mathbb{N}\mid i \leq$ number of faces in the polytope of $X \}$. We can index the set of faces by the set $ I$.   Call the inverse image of the  interior of the face $F_i$ as $X_{i}$. It can be easily seen that this gives a stratification of the
 orbifold where  each stratum $X_i$ is diffeomorphic to a complex torus.

%Using standard arguments, we can easily show that X is always stratified by
%locally closed suborbifolds $ X_{i}( i \in I)$( the interior of faces) such that the germs of the singularities of
%X along $ X_{i}$ are analytically isomorphic to that of a toric singularities
%defined by means of a (k-1)-dimensional lattice polytope $ \Delta_{F_{i}}$( where $F_{i}$ is the face determined  by$ X_{i}$) to that of a SL-toric
%singularity defined by means of a finite subgroup $ G_ {i}$ of $ SL(k,\mathbb{C})$,respectively, where
% where
%k denotes the  codimension of the face $F_ {i}$ in the polytope.

%\begin{remark}
It is easily seen that
 \begin{equation}\label{morestrat}
 W(\Delta_{F_{i}},uv)= \Sigma_{X_ {j} \geq  X_ {i}}  \widetilde{W}(\Delta_{F_{j}},uv)
 \end{equation}
 where $X_ {j} \geq X_ {i}$ if $\bar{X_ {j}} \supset X_i$. The above result is true because the coefficient of each term in the left hand side can  be broken in to ones
 with different heights (see definition $\eqref{ht}$, equations $\eqref{psi1},\eqref{lat2}$ and$\eqref{tpsi}$. The ones with height equal to the codimension of $X_ {i}$ contribute to  $ \widetilde{W}( \Delta_{F_{i}},uv))$. These come from $G_{F_i}^{\circ}$. Use the decomposition
$G_{F_i} = \bigsqcup_{F_j \supseteq F_i} G_{F_j}^{\circ}$ to observe
 that terms  with lesser heights correspond to  higher $X_ {j}$.
 %\end{remark}

\subsection{Poincar\'e Polynomial} Recall that
\begin{equation}\label{CR}
H^{\ast}_{CR}({\bf X}, \RR ) = \bigoplus_{F \le P} \bigoplus_{
g\in G_F^{\circ}} H^{\ast - 2 \iota(g)} (X(F), \RR)
\end{equation}
 where  $X(F)$ is  the inverse image of the face $F$.

 \begin{defn}
 The Poincar\'e polynomial of a cohomology of $X$ is a polynomial $P(X)(t)$ where the coefficient of  $t^d$ is the rank of the degree $d$ cohomology group.
  We denote by $PP(X)(v)$ the Poincar\'e polynomial of the ordinary singular cohomology and $PP_{CR}(X)(v)$ as the Poincar\'e polynomial of the Chen-Ruan cohomology of ${\bf X}$.
 \end{defn}
 Now if $X$ is a projective toric orbifold, it has  pure Hodge structure. Since the Zariski closure of the $ X_{i}$ are the  suborbifolds corresponding to the faces, from $\eqref{CR}$, $\eqref{psi1}$ and $\eqref{lat2}$, we have
 \begin{equation}\label{CR1}
PP_{CR}(X) (v)=
 \Sigma_{i \in I} PP(\bar{ X_{i}})(v)\widetilde{W}(\Delta_{F_{i}},v^{2}).
 \end{equation}
 %\begin{equation}
 %\Sigma_{i \in I} E(\bar{ X_{i}};v,v)\widetilde{W}(\Delta_{F_{i}},v^{2})=(PP_{CR})(X) (v).
 %\end{equation}
 %for those  $ X_i$ which  have elements $ g_i$ in their local groups whose heigths are same as their complex codimension of $ X_i$.
 %where $(PP_{CR})(X)(v)$    is poincare polynomial of chen ruan cohomology and $F_{i}$ face corresponding to $X_{i}$.

\subsection{\bf{Correspondence in quasitoric orbifolds}}
 Take a quasi-SL quasitoric orbifold ${\bf X}$.  A slight perturbation makes the polytope $P$ associated with the orbifold into a rational ploytope (see section 5.1.3.in \cite{[BP]}),  and with suitable dilations make it into an
 integral polytope $P'$ which is combinatorially equivalent to $P$. From the normal fan of $P'$ we get a projective toric orbifold
  $X'$ whose polytope is $P'$. (The orbifold structure of $X'$ is determined by its analytic structure and we may conveniently refrain
	from using bold-face notation.) Putting $ u=v$ in Theorem $\eqref{strata}$ we have
 \begin{equation}\label{stratification}
 E( X';v,v)=\Sigma_{i \in I } E(X'_i;v,v)
 \end{equation}
In  the left hand side   the  coefficient of $ v^{k}$  is the sum of $ e^{p,q}(X')$ where $p+q=k$. Since $X'$ is Kahler the Hodge structure is pure and  from remark  $\eqref{epq}$ it follows  $e^{p,q}(X')= (-1)^{p+q}h^{p,q}(X')$. Since  toric orbifolds  (see section 4 of \cite{[PS]}) have zero odd cohomology only the coefficient of  $ v^{2k}$  terms are nonzero.  By Baily's Hodge decomposition (see \cite{[WB]}), the Hodge numbers $h^{p,q}$ for
$p+q= 2k$ add up to the $2k$-th Betti number of singular cohomology group.
So the left hand side is the Poincar\'e polynomial of the ordinary cohomology, giving
 \begin{equation}
  PP( X')(v)=\Sigma _{i \in I } E(X'_i;v,v).
  \end{equation}
    It is known from section 4 of \cite{[PS]} that the Betti numbers  depend  on the combinatorial equivalence  class of the polytope $ P'$.
		As $P'$  is combinatorially equivalent  to $P$, the left hand side equals the Poincar\'e polynomial of the quasitoric orbifold $\bf{X}$.
		The right
 hand side is a  sum of $E$-polynomials of a number of tori. Since the number of tori of each dimension is the same by combinatorial equivalence of the polytopes, we have,

 \begin{equation}\label{newstratification}
PP(X)(v)=\Sigma _{i \in I } E(X_i,v,v)
\end{equation}
where $ \bigsqcup X_i$  is the stratification by tori of the quasitoric orbifold $\bf{X}$. Now from $\ref{CR}$  we get,
\begin{equation}\label{newpon_-1}
PP_{CR}(X)(v)=\Sigma _{i \in I } PP(\bar{ X_{i}})(v) \widetilde{W}( \Delta_{F_{i}},v^{2})
\end{equation}
 Using  $\ref{newstratification}$ we have
\begin{equation}\label{newpon_0}
PP_{CR}(X)(v)=\Sigma _{i \in I } \Sigma _{X_j \leq {X}_i}E(X_ {j},v,v) \widetilde{W}( \Delta_{F_{i}},v^{2})
 \end{equation}
Interchanging the order of summation, and using  $\ref{morestrat}$ we have
\begin{equation}\label{newpon}
 PP_{CR}(X)(v)=\Sigma _{j \in I}E(X_{j},v,v)W(\Delta_{F_{j}},v^{2})
\end{equation}

\begin{theorem}
Suppose ${\bf X}$ is a  quasi-SL quasitoric orbifold, and  $\hat{{\bf X}}$ a  crepant blowup. Then
 \begin{equation}
 PP_{CR}(X)(v) = PP_{CR}(\hat{X})(v))
  \end{equation}
\end{theorem}

 \begin{proof}
%From $\ref{newpon}$, we know that
 % \begin{equation}
%PP_{CR}(X)(v)=\Sigma _{j \in I}E(X_{j},v,v)W(\Delta_{F_{j}},v^{2})
 %\end{equation}

 Let $\rho: \hat{X}\rightarrow X$ be a  crepant blowdown. We set $ \hat{X_{i}}:=\rho^{-1}(X_{i})$. Then $\hat{X_{i}}$ has a natural
 stratification by products $ X_{i} \times ({(\mathbb{C}^{\ast })}^{codim(\theta)} )$ induced by the triangulation,
 \begin{equation}\label{triangulfi}
 \Delta_{F_{i}} =\cup_{\theta \in \tau_{i}} \theta
 \end{equation}
 where $\tau_{i}$ consists of all simplices  which intersect the interior of $ \Delta_{F_{i}}$,
 and $codim(\theta) $ denotes the codimension of $\theta$ in $\Delta_{F_i}$.

%The compactly supported cohomology of  $ {\mathbb{R}}^{2}-\{0\}$  is generated
% by $f(r)dr$ and $g(r,\Theta)dr d\Theta$ where $r$ and $\Theta$ usual polar coordinates of  ${\mathbb{R}}^{2}-\{0\}$, and the functions  $f$ and $g$ have compact support. Using Kunneth formula,

Note that the $E$-polyinomial of  a $k$-dimensional  complex torus is
$ (v^{2}-1)^{k}$.

From $\eqref{lati1}$ we have
	\begin{equation}
	 W(\Delta_{F_{i}}; v^{2}) = P_{\Delta_{F_{i}}}(v^{2}) (1-v^{2})^{d} \end{equation}
	where $d$ is the dimension of the face $ F_i$.
	Consider the triangulation \eqref{triangulfi} of $\Delta_{F_i}$.
	By counting lattice points using $\eqref{lati1}$ and applying the inclusion exclusion principle
	we have
	\begin{equation}
  P_{\Delta_{F_{i}}}(v^{2}) = \sum_{\theta \in \tau_i} {(-1)}^{codim(\theta)}P_{\theta}(v^{2}) =
	\sum_{\theta \in \tau_i} {(-1)}^{codim(\theta)} W (\theta,v^{2}) (1- v^2)^{-dim(\theta)}
	\end{equation}
	
	Multiplying both sides by $(1-v^2)^d$, we obtain
	 \begin{equation}\label{wdelta}
 W(\Delta_{F_{i}}; v^{2}) = \Sigma_{\theta \in \tau_{i}} ( v^{2}-1)^{codim(\theta)} W (\theta,v^{2})
  \end{equation}

%By counting lattice points, from $\eqref{lati}$ $\eqref{lati1}$ we have
 % \begin{equation}\label{wdelta}
 %W(\Delta_{F_{i}}; v^{2}) = \Sigma_{\theta \in \tau_{i}} ( v^{2}-1)^{codim(\theta)} W (\theta,v^{2})
 % \end{equation}

	%Now if we multiply the left hand side of the equation by $(1-v^{2})^{-d}$ (where $d$ is the dimension of
	%the face $ F_i$) from $\eqref{lati1}$ we have
	%$$(1-v^{2})^{-d}W(\Delta_{F_{i}}; v^{2})=P_{\Delta_{F_{i}}}(v^{2}).$$
	%Now taking $(1-v^{2})^{-d}$ inside the summation in the right hand side we get sum of
		%$$(v^{2}-1)^{codim(\theta)}(1-v^{2})^{-d} W (\theta,v^{2}) ={(-1)}^{codim(\theta)}(1-v^{2})^{-dim({\theta})}W (\theta,v^{2})
%	={(-1)}^{codim(\theta)}P_{\theta}(v^{2})$$  running over all $\theta \in \tau_i$. The right hand side is the count of lattice points by the %inclusion-exclusion principle of set theory.
	
%	For example  codimesion  zero simplices cover the whole polyhedron. So all lattice points are counted while taking their sum. Since a %codimension one simplex is intersection
% of two codimension zero simplices, the lattice points of this simplex is counted twice in the codimension zero count so they need to %subtracted. Hence the  equality holds.

 Since we are dealing with simplices $\theta $ which
	intersect the interior of $ \Delta_{F_{i}}$ each stratum of $\widehat{X}$ is counted once.
	This is because each stratum corresponds to the interior of a face and for each face we have a simplex and it will
	lie in the interior of exactly one of the original (pre-triangulation) polyhedrons.
 Thus the equation $\eqref{newpon}$ applied to $\widehat{X}$ gives
 \begin{equation}
PP_{CR}(\hat{X})(v)=\Sigma_{i\in I}E(X_{i};v,v) \Sigma_{\theta \in \tau_{i}}(v^{2}-1)^{codim(\theta)}W(\theta ;v^{2})
\end{equation}

Now using \eqref{wdelta}
  \begin{equation}
   PP_{CR}(\hat{X})(v)) = \Sigma_{i\in I}E(X_{i};v,v)  W(\Delta_{F_{i}}; v^{2}) = PP_{CR}(X)(v)
  \end{equation}
\end{proof}

%\newpage

\end{document}